\documentclass[12pt]{article}

\usepackage[top=1in, bottom=1in, left=1in, right=1in]{geometry}

\usepackage{amsmath,amssymb,amsthm,tikz,dsfont}
\usepackage{color}
\newcommand{\R}{\mathbb{R}}

\newcommand{\hx}{\hat{x}}

\newtheorem{lemma}{Lemma}[section]

\newtheorem{theorem}{Theorem}[section]

\newtheorem{proposition}{Proposition}[section]
\newtheorem{corollary}{Corollary}[section]
\newtheorem{remark}{Remark}[section]

\begin{document}

\title{Singularities Almost Always Scatter: Regularity Results for Non-scattering Inhomogeneities}
\author{Fioralba Cakoni and Michael S. Vogelius \footnote{Department of Mathematics, Rutgers University, New Brunswick,  New Jersey,  USA 
 (fc292@math.rutgers.edu) and  (vogelius@math.rutgers.edu)}}
\date{\empty}
\maketitle

 \begin{abstract}
 \noindent
In this paper we examine necessary conditions for an inhomogeneity to be non-scattering, or equivalently, by negation, sufficient conditions for it to be scattering. These conditions are formulated in terms of the regularity of the boundary of the inhomogeneity. We examine broad classes of incident waves in both two and three dimensions. Our analysis  is greatly influenced by the analysis carried out by Williams \cite{pom3} in order to establish that a domain, which does not possess the Pompeiu Property, has a real analytic boundary. That analysis, as well as ours, relies crucially on classical free boundary regularity results due to Kinderlehrer and Nirenberg \cite{pom000}, and Caffarelli \cite{pom00}.
 \end{abstract}

\noindent{\bf Key words:}  inverse scattering, inhomogeneous media,  non-scattering,  transmission eigenvalues, free boundary.\\
\noindent{\bf AMS subject classifications:} 35R30, 35J25, 35P25, 35P05

\section{Introduction}\label{form}
A perplexing question in scattering theory is whether there  are incoming time harmonic waves, at particular frequencies, that are not scattered by a given inhomogeneity, in other words the inhomogeneity is invisible to probing by such waves. We refer to wave numbers, that correspond to frequencies for which there exists a non-scattering incoming wave, as non-scattering.  The attempt to provide an answer to this question has led to the so-called transmission eigenvalue problem with the wave number as the eigen-parameter. This eigenvalue problem has a deceptively simple formulation, namely two elliptic PDEs in a bounded domain (representing the inhomogeneity)  with a single set of Cauchy data on the boundary. However, the problem is a non-selfadjoint eigenvalue problem with  challenging mathematical structure. The  non-scattering wave numbers form a subset of the real transmission eigenvalues. We refer the reader to the monograph \cite{CakoniColtonHaddar2016} for a comprehensive discussion of the transmission eigenvalue problem,  and to \cite{nguyen}, \cite{vodev} for the most up-to-date results on the spectral analysis for the scalar problem considered here.  A real transmission eigenvalue is not necessarily a non-scattering wave number, and it is desirable  to understand which (if any) are. Besides being mathematically appealing, this question is also important from an application point of view. In particular, at a non-scattering wave number the relative scattering operator \cite{res1} (otherwise known as the far field operator \cite{coltonkress}) is not injective. This causes the failure of some reconstruction methods for solving the inverse scattering problem.  Despite some progress made for special geometries \cite{liu, nonscat, nn5, ElH18, HSV16, vx}, the question of existence/non-existence of non-scattering wave numbers for a general inhomogeneity has been largely open until now. The main contribution of our  paper is that it provides finer necessary regularity conditions on the geometry of an inhomogeneity in order that it be non-scattering, or equivalently, it provides more general, sufficient conditions for it to be scattering. Let us proceed to formulate the specific scattering problem.

\vskip 10pt

\noindent
We consider Helmholtz scattering by an inhomogeneous medium of bounded support. We denote the inhomogeneity by $D\subset {\mathbb R}^m$, $m=2,3$,  and assume that $D$ is a bounded and simply connected region with a Lipschitz boundary $\partial D$. $\nu$ denotes the outward unit normal vector to $\partial D$ defined almost everywhere on $\partial D$. We assume that the inhomogeneity is situated in a homogeneous background and we denote by $n$ its refractive index. $n$ is a real valued function, with $n\in L^{\infty}(D)$, and   $n(x)\geq n_0>0$  almost everywhere in $D$. \footnote{The assumptions about simply connectedness of $D$, and strict positivity of $n$, could be considerably relaxed, but for the sake of clarity of exposition we have decided not to do so.} The propagation of a time harmonic monochromatic wave in homogeneous free space is modeled by the Helmholtz equation
\begin{equation}\label{inc}
\Delta v+k^2v=0
\end{equation}
where $k$ which is proportional to the frequency (e.g. in ${\mathbb R}^3$, $k=\omega/c_0$ where $c_0$ is  the sound speed of the homogeneous background). Let $q$ denote the function
$$
q(x)=\begin{cases} n(x) & ~x \in D\\
1& ~x \in \mathbb{R}^m \setminus \overline{D}~.
\end{cases} 
$$
We now formulate the {\it direct scattering  problem} for the inhomogeneous media $(D,n)$.  By an incident wave $v$ we understand a function that  satisfies (\ref{inc}) 
in  ${\mathbb R}^m$, except for possibly a subset of measure zero in the exterior of $\overline{D}$; this could be  a single point, for point sources, or a surface, for surface potentials. We decompose the total field as
$u_{tot}= u+v$, where $u$ represents the scattered field. The scattered field now satisfies
\begin{equation}\label{total}
\Delta u+k^2qu=k^2(1-q)v \qquad \qquad \mbox{in}\;\; {\mathbb R}^m~,
\end{equation}
along with the  outgoing Sommerfeld radiation condition
\begin{equation}\label{somer}
\displaystyle\lim_{r\to\infty}r^{\frac{m-1}{2}}\left(\frac{\partial u}{\partial r}-iku\right)=0~,
\end{equation}
as  $r:=|x|\to \infty$  uniformly with respect to $\hat{x}:=x/|x|$   (see e.g \cite{coltonkress}). 
\noindent
The scattered field $u$, which is in $H^2_{loc}({\mathbb R}^m)$, assumes the  following asymptotic behavior  as $r:=|x|\to \infty$
$$
u(x)=\frac{\exp(ikr)}{r^{\frac{m-1}{2}}}u^\infty(\hx)+O\left(r^{-\frac{m+1}{2}}\right)~,
$$
where the function $u^\infty(\hx)$ defined on the unit $m-1$ sphere is called the far field pattern of the scattered field $u$. Rellich's Lemma  (see e.g \cite{coltonkress}) states that the vanishing of $u^\infty(\hat x)$ on the unit $m-1$ sphere for all directions $\hat x$ implies vanishing of the scattered field $u(x)=0$ for all $x\in {\mathbb R}^m\setminus \overline{D}$. A natural question is:  given the inhomogeneity $(D, n)$, does there exist a wave number $k>0$ and an incident wave $v$ such that the far field pattern $u^\infty:=u_v^\infty$  of the corresponding solution of  (\ref{total}) is identically zero? Such an incident field is referred to as a non-scattering incident wave and the corresponding $k>0$ as a non-scattering wave number. The non-scattering phenomenon is equivalent to the existence of $u\in H^2_0(D)$ solving
\begin{equation}\label{nsp}
\Delta u+k^2nu=k^2(1-n)v \qquad \qquad \mbox{in}\;\; D
\end{equation}
where 
$$H^2_0(D):=\left\{u\in H^2(D), \mbox{ such that } \, u=\displaystyle{\frac {\partial u}{\partial \nu}}=0 \; \mbox{ on }\, \partial D\right\}~.\footnote{ $H^2_0(D)$ is the closure of $C_c^\infty(D)$ in $H^2(D)$ -- for a Lipschitz domain it coincides with those $H^2(D)$ functions, which when extended by zero outside $D$ remain in $H^2(\mathbb{R}^m)$. For a Lipschitz domain $H^2_0(D)$ also coincides with those functions $u \in H^2(D)$ for which $u$ and $\frac{\partial u}{\partial \nu}$ (defined in the sense of traces) vanish on $\partial D$.}
$$  
\noindent
If we include $v$, and focus on $D$, this may then be written
\begin{eqnarray}
\Delta u+k^2nu=k^2(1-n)v &\qquad \qquad \mbox{in}&\;\; D~,\label{te1}\\
\Delta v+k^2v=0 \qquad &\qquad \qquad \mbox{in}&\;\; D~,\\
u=\displaystyle{\frac {\partial u}{\partial \nu}}=0 \qquad &\qquad \qquad \mbox{on}&\;\; \partial D\label{te2}~,
\end{eqnarray}
with $u\in H^2(D)$. The equations (\ref{te1})-(\ref{te2}), with the requirement that $v\in L^2(D)\setminus \{0\}$ and $u\in H^2(D)$, are equivalent to the statement that $k$ is a transmission eigenvalue, with corresponding eigenvector $(u,v)$. $k$ being a transmission eigenvalue is therefore a necessary condition for $k$ being a non-scattering wave number, corresponding to the incident wave $v$ (defined on all of $\mathbb{R}^m$). Whereas transmission eigenvalues exist for quite general (non-smooth) domains, the results we establish in this paper show that the existence of non-scattering wave numbers (for regular $n$) imply some degree of regularity of the boundary $\partial D$ (for quite general incident waves). In terms of the transmission eigenvalue problem (\ref{te1})-(\ref{te2}), formulated only on $D$, similar regularity results would follow if we were to insist that $v$ be appropriately regular up to the boundary (on top of being in $L^2(D)$). This may also be seen as a reflection of the fact that a transmission eigenvalue only is a non-scattering wave number if the corresponding $v \in L^2(D)$ can be extended to a solution of the Helmholtz equation in the exterior of $D$. We recall that, if $n-1$ is of one sign in a neighborhood of the boundary $\partial D$, then  the set of transmission eigenvalues (possibly complex) is at most discrete with infinity as the only possible accumulation point. Furthermore if $n-1$ is of one sign in the entire $D$, then there exists an infinite sequence of real transmission eigenvalues (see \cite{1} and also \cite{CakoniColtonHaddar2016}). However, this paper concerns the existence of non-scattering wave numbers, and our approach does not require any knowledge about the spectrum of the transmission eigenvalue  problem.

\vskip 10pt

\noindent
In the case  of spherically symmetric media, i.e., when $D$ is a ball of radius $R$ centered at the origin, and $n:=n(r)$ depends only on the radial variable,   it is possible to show by separation of variables  that $k>0$ being a transmission eigenvalue is a necessary and sufficient condition for $k$ being a non-scattering wave number. Furthermore, the non-scattering wave numbers form an infinite discrete set with $+\infty$ as the only accumulation point. In this case the non-scattering  incident waves are superposition of plane waves, otherwise known as Herglotz wave functions, with particular densities; each density associated with an infinite set of  non-scattering wave numbers (see e.g. \cite[Chapter 10]{coltonkress} and \cite{vx}). The spherically  symmetric configuration is unstable  with respect to non-scattering. Vogelius and Xiao  in \cite{vx}  have shown  in ${\mathbb R}^2$, and for constant refractive index $n\neq 1$, that  if the disk is perturbed even slightly to a proper ellipse with arbitrarily small eccentricity, then there exist at most finitely many positive wave numbers for which a  Herglotz wave  function with a fixed, smooth, non-trivial density can be non-scattering. When the boundary of the inhomogeneity $D$ contains a corner, then the set of non-scattering wave numbers is empty, provided $n-1\neq 0$ at the corner, and under some  local regularity  on $n$. This result was first proven by  Bl{\aa}sten, P\"aiv\"arinta   and Sylvester in \cite{nonscat}  for a right corner, followed by \cite{PSV17} for a convex corner, and their analysis employs the so-called complex geometric solutions for the Helmholtz equation. This approach is generalized in \cite{nn5} to the scattering problem with a general operator of divergence form  (instead of the Laplacian). The most comprehensive analysis, implying that corner and edge singularities always scatter in ${\mathbb R}^m$,  is given  by Elschner and Hu in  \cite{ElH18}, based on a refined corner singularity analysis of the solution to (\ref{nsp}).  All these scattering results for geometries with corners are valid without any restrictions on the incident wave, and they provide the  foundation for  proving that  a convex polygonal inhomogeneity is uniquely determined from scattering data corresponding to one single incident wave \cite{HSV16}. Up to now there is a large gap, between spherically symmetric media and inhomogeneities containing a corner, in which little is known about non-scattering. In fact, nothing is known for general smooth domains $D$, with the exception of the partial results  in \cite{vx}  and \cite{liu}. In \cite{vx}  it is shown that given any smooth, strictly convex  domain in ${\mathbb R}^2$, there exist at most finitely many positive wave numbers $k$ for which an (arbitrary but fixed)  incident plane wave  can be non-scattering. In \cite{liu} the authors  prove that  inhomogeneities, containing a boundary point of high curvature (near which the inhomogeneity could be analytic) scatter any  incident field, whose modulus is bounded away from zero by a constant depending on the curvature and the value of the contrast $n-1$ at this point. Our paper substantially contributes to filling this gap.

\vskip 10pt
\noindent
Our main results are stated precisely in the next section.  Roughly speaking,  under the global  assumptions of a Lipshitz boundary $\partial D$, and $n\in L^\infty(D)$,  we show that  if there is a point  $x_0\in \partial D$,  such that $n$ is analytic in a neighborhood of $x_0$, but the boundary is not  analytic in any neighborhood of $x_0$, then every  incident field $v$ is scattered, provided $(n(x_0)-1)v(x_0)\neq 0$.  We establish a similar result for $n$ that are less regular locally near $x_0$, but still in $C^{1,1}$. In this case we show that if the boundary is not sufficiently smooth locally (related to the order of smoothness of $n$), then every incident  wave is  scattered,  again provided $(n(x_0)-1)v(x_0)\neq 0$. Although our results address domains with corners (in fact we only require Lipschitz boundaries) they are only valid provided the incident field $v$ is non-vanishing at $x_0$ (here: the corner). This is not required by the prior results on corners,  mentioned above. A non-vanishing condition on the incident fields  is essential to our approach. Such non-vanishing holds for plane waves or point sources; it is unclear exactly what limitations this imposes for generic Herglotz waves or for generic real analytic solutions to the  Helmholtz equation. However, in ${\mathbb R}^2$ and for  analytic refractive index $n$ near the boundary, we prove that there exist at most finitely many positive wave numbers $k$ for which one may find a non-scattering Herglotz wave function (with density in a fixed compact subset of $C^1$), unless $\partial D$ is almost everywhere analytic (see Section \ref{appl} for the appropriate definition). 

\vskip 10pt
\noindent
As a direct consequence of the proof of our main results, we conclude that at a transmission eigenvalue, the part $v$ of the transmission eigenfunction lacks sufficient  regularity near a singular boundary point $x_0$, unless it vanishes at this point, thus recovering similar results  for the case of  corners in \cite{nonscat4} and  \cite{BLLW17}.  For a precise statement of this result, see the end of Section \ref{proof}. Our analysis also yields results concerning the regularity of the support of non-radiating sources \cite{kusiak}. These results are discussed in the last section of this paper; they substantially generalize similar results in \cite{nonscat4} for sources whose support contains corner singularities.

\vskip 10pt
\noindent
At the core, our analysis relies on viewing  the boundary with vanishing Cauchy data as a free boundary, and applying the free boundary regularity results of  Caffarelli \cite{pom00},   and Kinderlehrer  and Nirenberg \cite{pom000} for second order elliptic equations. There is a striking similarity in the mathematical structure of the problem of non-scattering inhomogeneities, and the problem of domains that do not possess the  Pompeiu property \cite{pom0, pom2, pom1}.  Regularity properties of the latter are established by Williams \cite{pom3}, and the analysis here in several places borrows significantly from his original work.

\section{Statement of our main results}
In this section we state the main results of this paper. These results are stated in terms of sufficient conditions of non-smoothness of $\partial D$ for scattering to occur for a given incident wave. By negation they could as well have been stated as necessary smoothness conditions that follow from non-scattering. In the formulation the incident wave is a solution to
\begin{equation}\label{incB}
\Delta v+k^2v=0 
\end{equation}
in $\mathbb{R}^m$, except possibly a set of measure zero, external to $\overline{D}$. Actually it suffices that $v$ is a solution to this equation in $D$, with $v$ real analytic on $\overline D$.   In the formulation of our main results we also refer to the region $D_\delta\subset D$, defined  by 
$$D_\delta:=\left\{x\in D, \; dist(x,\partial D)<\delta\right\}\;\; \;\;\; \mbox{for some fixed $\delta>0$}.$$
The proof of these results is postponed to Section \ref{proof}.

\begin{theorem}\label{scatnon1} Let $k>0$ be a fixed wave number. Assume the positive refractive index  $n$ is in $L^\infty(D)$, and that the boundary $\partial D$ is  Lipschitz. Consider an incident field $v$ satisfying (\ref{incB}).  Assume  that  $n$ is real analytic in $\overline{D_\delta}$, and there exists $x_0\in\partial D$   such $(n(x_0)-1)v(x_0)\neq 0$. Assume furthermore that $\partial D\cap B_r(x_0)$ is not real analytic for any ball $B_r(x_0)$ of radius $r$ centered at $x_0$.  Then the incident field $v$ is scattered by the inhomogeneity $(D, n)$. In other words: there exists no $H^2_0(D)$ solution to (\ref{nsp}).
\end{theorem}

\noindent
 For less regular refractive index $n$ there is a similar result. 
\begin{theorem}\label{scatnon2} Let $k>0$ be a fixed wave number. Assume the positive refractive index $n$ is in  $L^\infty(D)$, and that the boundary $\partial D$ is  Lipschitz. Consider an incident field $v$ satisfying (\ref{incB}).  Assume that $n\in C^{m, \mu}(\overline{D_\delta})\cap C^{1,1}(\overline{D_\delta})$  for $m \geq 1$, $0<\mu<1$, and there exists $x_0\in \partial D$ such that $(n(x_0)-1)v(x_0)\neq 0$. Assume furthermore that  $\partial D\cap B_r(x_0)$  is not of class $C^{m+1, \mu}$, for any ball $B_r(x_0)$ of radius $r$ centered at $x_0$.  Then the incident field $v$ is scattered by the inhomogeneity $(D, n)$. In other words: there exists no $H^2_0(D)$ solution to (\ref{nsp}).
\end{theorem}

\noindent
\begin{remark}
The smoothness assumptions on the refractive index $n$ in Theorem \ref{scatnon1} and Theorem \ref{scatnon2} are only needed locally  in $\overline{D}\cap B_R(x_0)$ for some ball centered at $x_0$ of radius $R>0$.
\end{remark}

\noindent
Of course Theorems \ref{scatnon1} and \ref{scatnon2} only add insight if the wave number $k$ is a real transmission eigenvalue (which is a necessary condition for the incident field to produce a vanishing scattered field). At any  $k$ other than a transmission eigenvalue, every incident field is scattered by the given inhomogeneity. However, it is important to emphasize that we do not need to know a priori that $k>0$ is a transmission eigenvalue, and therefore our results hold under weaker conditions on the contrast than those (currently) needed to prove the existence  of real transmission eigenvalues. If $k>0$ is a transmission eigenvalue, the assumptions in Theorems \ref{scatnon1} and \ref{scatnon2} imply that the part $v$ of the transmission eigenfunction (\ref{te1})-(\ref{te2}) cannot be extended into the exterior of $D$ as solution of the Helmholtz equation, provided $n\neq 1$ on $\partial D$ and that this eigenfunction does not vanish at the point $x_0\in \partial D$ (see Corollary \ref{cor1}).

\noindent

\section{A free boundary regularity result}
With  $a(x) = k^2 n(x)$, and $b(x)= k^2(1-n(x))v(x)$, the problem (\ref{nsp}) becomes
\begin{eqnarray}
&\Delta u+a(x)u=b(x)&\;\;\;  \mbox{in}\;\; D \label{fbp}\\
&u=\displaystyle{\frac {\partial u}{\partial \nu}}=0 \qquad & \mbox{on}\;\; \partial D. \label{fbp1}
\end{eqnarray}
In order to prove our main results we shall make use of two classical free boundary regularity results.  The first result  is due to Kinderlehrer and Nirenberg  in  \cite[Theorem 1' on page 377]{pom000}. In \cite{pom000} the Theorem is proven for more general nonlinear second order elliptic partial differential operators, but in the following  we state it as it applies to our linear equation (\ref{fbp}). 

\begin{theorem}\label{th1} Suppose that $0\in \partial D$, and $\partial D\cap B_R(0)$ is of class  $C^1$ for some ball  $B_R(0)$ of radius $R$ centered at  $0$. Suppose $a$ and $b$ are real valued functions in $ C^{1}(\overline {D}\cap B_R(0))$, with $a(0)\neq 0$ and $b(0)\neq 0$. Furthermore suppose there exists a real valued solution $u$ to (\ref{fbp})-(\ref{fbp1}), with $u\in C^2({\overline D}\cap B_R(0))$. Then
\begin{enumerate}
\item $\partial D\cap B_{R'}(0)$ is of class $C^{1, \alpha}$ for every positive $\alpha<1$, and some $R'<R$.
\item  If additionally $a\in C^{m, \mu}(\overline {D}\cap B_R(0))$ and $b\in C^{m, \mu}(\overline {D}\cap B_R(0))$ for $m\geq 1$, $0<\mu<1$ then $\partial D\cap B_{R'}(0)$ is of class $C^{m+1, \mu}$, for some $R'<R$.
\item  If $a$ and  $b$ are real analytic in $\overline {D}\cap B_R(0)$ then $\partial D\cap B_{R'}(0)$ is real analytic for some $R'<R$.
\end{enumerate}
\end{theorem}
\begin{remark}\label{rem1}
The regularity of the free boundary is a local property. Correspondingly, the result of Theorem \ref{th1} holds for $u$ solving  (\ref{fbp}) in $D \cap B_R(0)$ with zero Cauchy data (\ref{fbp1}) only on the part of boundary $\partial D\cap B_R(0)$. However, in our particular applications the solution $u$ will be defined on all of $D$. 
\end{remark}

\noindent
In this paper we initially assume that  $\partial D$ is only Lipschitz regular. In order to apply Theorem \ref{th1} we must first show that the free boundary $\partial D\cap B_R(0)$ is indeed $C^1$, and then verify that the solution $u$ to (\ref{fbp})-(\ref{fbp1}) is in $C^2({\overline D}\cap B_R(0))$.  This intermediate regularity is achieved with the help of a  classical result on regularity of the free boundary due to Caffarelli \cite[Section 1.2 and Theorem 3 on page 166]{pom00}, which we state in the following theorem, modified to the framework of our problem. This result refers to  a function $w$  that satisfies 
\begin{equation}\label{caf}
\Delta w=g  \quad \mbox{in } \, D\cap B_R(0), \qquad \mbox{such that $w=\displaystyle{\frac{\partial w}{\partial \nu}}=0 \quad \mbox{on } \, \partial D\cap B_R(0)$}~,
\end{equation}
where again  $0\in \partial D$ and  $B_R(0)$ is some ball of radius $R$ centered at  $0$. 
\begin{theorem}\label{th2} Suppose that $\partial D\cap B_R(0)$ is Lipschitz and the function $w$ satisfying (\ref{caf}) is in $C^{1,1}(\overline{D}\cap B_R(0))$. Furthermore, assume that  $w\leq 0$ in $D\cap B_R(0)$, and  $g$ has a $C^1$-extension $g^*$ in a neighborhood of  $\overline{D}\cap B_R(0)$  such that $g^*\leq -\alpha<0$. Then there exists $R'<R$ such that $\partial D\cap B_{R'}(0)$ is of class $C^1$ and all second derivatives of $w$ are continuous up to $\partial D\cap B_{R'}(0)$, {\it i.e.}, $w \in C^2(\overline{D}\cap B_{R'}(0))$.
\end{theorem}

\noindent
The first obstacle to the application of Theorem \ref{th2} is to verify that the $H^2_0(D)$ solution $u$ to (\ref{nsp})  has all second derivatives  uniformly bounded in $D\cap B_R(x_0)$.  For this purpose, we next investigate the regularity of the volume potential.

\section{A regularity result for the volume potential}\label{secreg}
Let $\Phi(x,y)$ be the free space fundamental solution to the Laplace operator,  given by
\begin{equation}\label{fund}
\Phi(x,y):=\left\{\begin{array}{rrcll}\displaystyle{\frac{1}{4\pi |x-y|}} \quad \; & \qquad  \mbox{in }\, {\mathbb R}^3 \\
& \\
\displaystyle{\frac{1}{2 \pi}}\ln \displaystyle{\frac{1}{|x-y|}} & \qquad \mbox{in }\, {\mathbb R}^2 ~.
\end{array}\right.
\end{equation}
For later use we note the following estimates:
$$\hbox{for } 1\le j \le m~~ \left|\frac{\partial \Phi}{\partial x_j}(x,y)\right|\leq \frac{C}{|x-y|^{(m-1)}} ~~~~ \hbox{ in } \mathbb{R}^m~, m=2,3~,$$
$$\hbox{for } 1\le i,j \le m~~\left|\frac{\partial^2 \Phi}{\partial x_i\partial x_j}(x,y)\right|\leq \frac{C}{|x-y|^m}~~~ \hbox{ in } \mathbb{R}^m~, m=2,3~.$$

\noindent 
In this section we study the regularity of the (weighted) volume potential
$$w_\psi(x)=\int_D\psi(y)\Phi(x,y)\,dy~.$$
The following lemma is proven in  \cite[Lemma 3.7]{kirsch} in  the case of ${\mathbb R}^3$. For completeness we include here the proof in ${\mathbb R}^m$, $m=2,3$.
\begin{lemma}\label{c1}
For $\psi\in L^\infty(D)$ we have that $w_\psi\in C^1({\mathbb R}^m)$, $m=2,3$ and
$$\frac{\partial w_\psi}{\partial x_j}(x)=\int_D \psi(y)\frac{\partial \Phi}{\partial x_j}(x,y)\,dy, \qquad x\in {\mathbb R}^m, \; j=1,\cdots m~.$$
\end{lemma} 
\begin{proof}
It easy to see that $w_\psi$ is in $C^0(\mathbb{R}^m)$. Now, consider a smooth cut-off function $\xi$ such that  $0\leq \xi(t)\leq 1$, $\xi(t)=1$ for $t\geq 2$ and $\xi(t)=0$ for $t\leq 1$ and set
$$d_j(x):= \int_D \psi(y)\frac{\partial \Phi}{\partial x_j}(x,y)\,dy~,$$ 
which exists from the above estimates of the derivatives  of the fundamental solutions.
Next, let us denote by 
$$w_\epsilon(x):=\int_D \psi(y)\xi(|x-y|/\epsilon)\Phi(x,y)\,dy~;$$
notice that $w_\epsilon \in C^\infty({\mathbb R}^m)$ and that $w_\epsilon \rightarrow w_\psi$ in $C^0(\mathbb{R}^m)$ as $\epsilon \rightarrow 0$. We have 
$$d_j(x)-\frac{\partial w_\epsilon}{\partial x_j}(x)=\int_D \psi(y)\frac{\partial}{\partial x_j}\left\{\Phi(x,y)\left[1-\xi(|x-y|/\epsilon)\right]\right\}\,dy~,$$
and so in ${\mathbb R}^3$ we estimate
\begin{eqnarray*}
\left|d_j(x)-\frac{\partial w_\epsilon}{\partial x_j}(x)\right|&\leq& \|\psi\|_\infty\int_{|y-x|\leq 2\epsilon}\left(\left|\frac{\partial \Phi}{\partial x_j}(x,y)\right|+\frac{\|\xi'\|_\infty}{\epsilon}|\Phi(x,y)|\right)\,dy\\
&\leq&C\int_0^{2\epsilon}\left(\frac{1}{r^2}+\frac{1}{\epsilon r}\right)r^2 \,dr=C_1\epsilon~,
\end{eqnarray*}
whereas in ${\mathbb R}^2$
\begin{eqnarray*}
\left|d_j(x)-\frac{\partial w_\epsilon}{\partial x_j}(x)\right|&\leq& \|\psi\|_\infty\int_{|y-x|\leq 2\epsilon}\left(\left|\frac{\partial \Phi}{\partial x_j}(x,y)\right|+\frac{\|\xi'\|_\infty}{\epsilon}|\Phi(x,y)|\right)\,dy\\
&\leq&C\int_0^{2\epsilon}\left(\frac{1}{r}+\frac{1}{\epsilon}\ln\frac{1}{r} \right)r \,dr\leq C_2\epsilon \ln \frac{1}{\epsilon}~. \end{eqnarray*}
In both cases $\frac{\partial w_\epsilon}{\partial x_j}\to d_j$ uniformly in ${\mathbb R}^m$ as $\epsilon \to 0$, which shows that $w_\psi\in C^1({\mathbb R}^m)$ and  that $\displaystyle{\frac{\partial w_\psi}{\partial x_j}}=d_j$. This completes the proof.
\end{proof}

\noindent
All second derivatives of $w_\psi$  exist for $x\in {\mathbb R}^m\setminus \overline{D}$, and one can differentiate twice inside the integral to obtain
\begin{eqnarray}
\frac{\partial^2 w}{\partial x_i \partial x_j}(x)&=&\int_D\psi(y)\frac{\partial^2\Phi}{\partial x_i \partial x_j}(x,y)\,dy \nonumber \\
&=&\int_D\left[\psi(y)-\psi(x)\right]\frac{\partial^2\Phi }{\partial x_i \partial x_j}(x,y)\,dy +\psi(x)\int_D\frac{\partial^2\Phi}{\partial x_i \partial x_j}(x,y)\,dy \nonumber\\ 
&=&\int_D\left[\psi(y)-\psi(x)\right]\frac{\partial^2\Phi}{\partial x_i \partial x_j}(x,y)\,dy -\psi(x)\int_{\partial D}\frac{\partial \Phi}{\partial x_j}(x,y) \nu_i(y)\,dy~, \label{dder}
\end{eqnarray}
provided $\psi$ extends into $\mathbb{R}^m\setminus \overline{D}$.
Here the last integral over $\partial D$ is obtained by using the divergence theorem (the minus sign arises when one replaces an $x_i$ derivative with a $y_i$ derivative).  Note that unit normal vector $\nu=\left(\nu_i\right)_{i=1,m}$ is well-defined for almost all $y\in \partial D$. We  show next that if $\psi 
$, in addition to being bounded on $D$, is in $C^{\alpha}( B_R(0))$, then (\ref{dder}) holds true for $x\in  D\cap B_R(0)$. To this end,  we set
$$d_{ij}(x):=\int_D\left[\psi(y)-\psi(x)\right]\frac{\partial^2\Phi}{\partial x_i \partial x_j}(x,y)\,dy -\psi(x)\int_{\partial D}\frac{\partial \Phi}{\partial x_j}(x,y) \nu_i(y)\,dy~.$$ Note that $d_{ij}(x)$ is well defined for $x\in D\cap B_R(0)$, since for $\psi \in C^{\alpha}(B_R(0))$  the integrand inside the volume integral behaves as
\begin{equation}\label{gr}
\left|\left[\psi(y)-\psi(x)\right]\frac{\partial^2\Phi}{\partial x_i \partial x_j}(x,y)\right|\leq C|x-y|^{\alpha-m}~,~ \hbox{ for } y \hbox{ near } x,
\end{equation}
and is bounded for $y$ away from $x$;
the surface integral exists since $x$ is not on  $\partial D$.  Now we  choose $2\epsilon<dist(x, \partial D)$ and  again consider a smooth cut-off function $\xi$ such that  $0\leq \xi(t)\leq 1$, $\xi(t)=1$ for $t\geq 2$ and $\xi(t)=0$ for $t\leq 1$. Set
$$d_{j,\epsilon}(x):=\int_D \psi(y)\xi(|x-y|/\epsilon)\frac{\partial \Phi}{\partial x_j}(x,y)\,dy~.$$
We obtain
$$\frac{\partial d_{j,\epsilon}}{\partial x_i}(x)=\int_D\left[\psi(y)-\psi(x)\right]\frac{\partial }{\partial x_i}\left(\xi(|x-y|/\epsilon)\frac{\partial \Phi}{\partial x_j}(x,y)\right)dy -\psi(x)\int_{\partial D}\frac{\partial \Phi}{\partial x_j}(x,y) \nu_i(y)\,dy~,$$
and therefore
\begin{eqnarray*}
\left|d_{ij}(x)-\frac{\partial d_{j,\epsilon}}{\partial x_i}(x)\right|&\leq& C\int_{|y-x|\leq 2\epsilon}\left(\frac{1}{|y-x|^m}+\frac{\|\xi'\|_\infty}{\epsilon|y-x|^{m-1}}\right)|y-x|^\alpha dy\\
&=&C\int_0^{2\epsilon}\left(\frac{1}{r^{1-\alpha}}+\frac{\|\xi'\|_\infty}{\epsilon} r^\alpha \right) \,dr\leq C\epsilon^\alpha~.
\end{eqnarray*}
Hence, as $\epsilon \to 0$, $d_{j,\epsilon}(x)$ converges to $\frac{\partial w_\psi}{\partial x_j}(x)$, and  $\displaystyle{\frac{\partial d_{j,\epsilon}}{\partial x_i}(x)}$ converges to  $d_{ij}(x)$, both uniformly on compact subsets of $D \cap B_R(0)$. Thus  $d_{ij}(x)=\displaystyle{\frac{\partial^2 w_\psi}{\partial x_i \partial x_j}(x)}$  for $x\in B_{R}(0)\cap D$.

\noindent
Even for smooth $\psi$,  but with $\psi\neq 0$ on $\partial D$, the second derivatives of $w_\psi$ maybe become unbounded as $x$ approaches a boundary point from either inside or outside $D$.  Thus the volume potential is not necessarily in $C^2(\overline{D}\cap B_R(0))$. However, we can show that symmetric jumps of the second derivative (to become precise later) are uniformly bounded near  $0\in \partial D$,   when $\psi \in C^{\alpha}(B_R(0))$ for $0<\alpha<1$.  A similar result is  proven in \cite[Theorem 2]{pom3} for $\psi\equiv 1$. Our method of proof of Lemma \ref{c2} below is in many ways very similar to that in \cite{pom3}. We provide the details for completeness.

\noindent
First we introduce some notations. Denote $x:=(x^{(m-1)}, x_m)\in {\mathbb R^{m}}$ where $x^{(m-1)}\in {\mathbb R}^{m-1}$, and consider a cylindrical neighborhood  of $0$ defined by $N:=N(\rho,h)=B_\rho^{(m-1)}(0)\times [-h, h]$, where  $B_\rho^{(m-1)}(0)$ is the $m-1$ dimensional ball of radius $\rho$ centered at the origin. We assume that $B_{2r}(0)\subset N\subset \overline{N}\subset B_R(0)$. Furthermore, we assume (by appropriate rotation and selection of $\rho$ and $h$) that $N\cap \partial D$ is the graph $x_m=f(x^{(m-1)})$ of a Lipshitz continuous function $f:B_\rho^{(m-1)}(0)\to {\mathbb R}$, with Lipschitz constant $K$. We also assume that $h>K\rho$ and 
$$
N\cap D=\left\{(x^{(m-1)}, x_m)~:~ \, x^{(m-1)}\in B_\rho^{(m-1)}(0), \, f(x^{(m-1)})<x_m<h\right\},
$$
Finally we denote by  $e_m$ the unit vector in the  $m$-direction. We can now prove the following lemma.
\begin{lemma}\label{c2}
Assume that $\psi \in C^{\alpha}(B_R(0))$, for $0<\alpha<1$, in addition to being bounded on $D$. Then there exist $0<r$ so that the symmetric jumps 
$$\frac{\partial^2 w_\psi}{\partial x_i \partial x_j}(x+\eta e_m)-\frac{\partial^2 w_\psi}{\partial x_i \partial x_j}(x-\eta e_m), \qquad 1\leq i,j\leq m$$
across  the boundary at $x$ are uniformly bounded with respect to $0<\eta\leq r$ and $x\in \partial D\cap B_r(0)$.
\end{lemma}
\begin{proof} Using (\ref{dder}), outside and inside $D$, we write 
\begin{eqnarray*}
&&\frac{\partial^2 w_{\psi}}{\partial x_i \partial x_j}(x+\eta e_m)-\frac{\partial^2 w_{\psi}}{\partial x_i \partial x_j}(x-\eta e_m)\nonumber \\
&&\qquad =\int_D\left[\psi(y)-\psi(x+\eta e_m)\right]\frac{\partial^2\Phi}{\partial x_i \partial x_j}(x+\eta e_m,y)\,dy\nonumber \\
&&\qquad -\int_D\left[\psi(y)-\psi(x-\eta e_m)\right]\frac{\partial^2\Phi}{\partial x_i \partial x_j}(x-\eta e_m,y)\,dy \nonumber \\
&&\qquad -\psi(x+\eta e_m)\int_{\partial D}\frac{\partial \Phi}{\partial x_j}(x+\eta e_m,y) \nu_i(y)\,ds_y+\psi(x-\eta e_m)\int_{\partial D}\frac{\partial \Phi}{\partial x_j}(x-\eta e_m,y) \nu_i(y)\,ds_y
\end{eqnarray*}
for $x\in \partial D\cap B_r(0)$.  In the above integral expressions the part of the integrals taken over $D\setminus B_R(0)$ and $\partial D\setminus B_R(0)$ are uniformly bounded with respect to $\eta$  in $[0,r]$ and for all $x\in \partial D\cap B_r(0)$. So it suffices to consider only the integrals over $B_R(0)\cap D$  and $B_R(0)\cap \partial D$.  Next we have the following estimates for the integrands of the volume integrals
$$\left|\left[\psi(y)-\psi(x\pm\eta e_m)\right]\frac{\partial^2\Phi}{\partial x_i \partial x_j}(x\pm\eta e_m,y)\right|\leq C|x\pm\eta e_m-y|^{\alpha-m}~,~~~y \in B_R(0)$$
for $x\in \partial D\cap B_r(0)$ and $\eta<r$ (note that $x\pm\eta e_m\in B_{2r}(0)\subset B_R(0)$). Therefore the integrals over $D\cap B_R(0)$ are bounded uniformly in $\eta\in [0, r]$ and $x\in \partial D\cap B_r(0)$.  Next we consider the boundary integral  terms
$$\psi(x+\eta e_m)\int_{\partial D\cap B_R(0)}\frac{\partial \Phi}{\partial x_j}(x+\eta e_m,y) \nu_i(y)\,ds_y - \psi(x-\eta e_m)\int_{\partial D\cap B_R(0)}\frac{\partial \Phi}{\partial x_j}(x-\eta e_m,y) \nu_i(y)\,ds_y$$
for $1\leq i,j\leq m$. The above expression can be written  as
$${\mathbb I}_1+{\mathbb I}_2+{\mathbb I}_3$$
where
$${\mathbb I}_1:=\int_{\partial D\cap B_R(0)}\left[\psi(x+\eta e_m)- \psi(y)\right]\frac{\partial \Phi}{\partial x_j}(x+\eta e_m,y) \nu_i(y)\,ds_y$$
$${\mathbb I}_2:=\int_{\partial D\cap B_R(0)}\left[\psi(y)- \psi(x-\eta e_m)\right]\frac{\partial \Phi}{\partial x_j}(x-\eta e_m,y) \nu_i(y)\,ds_y$$
and 
$${\mathbb I}_3:=\int_{\partial D\cap B_R(0)}\left[\frac{\partial \Phi}{\partial x_j}(x+\eta e_m,y)-\frac{\partial \Phi}{\partial x_j}(x-\eta e_m,y) \right]\psi(y)\nu_i(y)\,ds_y.$$
Using the fact that  $\psi \in C^{\alpha}(B_R(0))$, and that $\tilde x:=x\pm\eta e_m\in B_{2r}(0)\subset N$  for $x\in \partial D\cap B_r(0)$ and $\eta<r$, we obtain
\begin{eqnarray*}
\left|{\mathbb I}_{1,2}\right|&\leq &C_1\int_{\partial D\cap B_R(0)}\frac{1}{|(x\pm\eta e_m)-y|^{m-1-\alpha}}\,ds_y=C+C_1\int_{\partial D\cap N}\frac{1}{|\tilde x-y|^{m-1-\alpha}}\,ds_y\\
& \leq &C+C_2\int_{B^{(m-1)}_\rho(0)}\frac{1}{|\tilde x^{(m-1)}-y^{(m-1)}|^{m-1-\alpha}}\sqrt{1+|\nabla f(y^{(m-1)}|^2}\,dy^{(m-1)}\\
&\leq & C+ C_3\int_{B^{(m-1)}_\rho(0)}\frac{1}{|\tilde x^{(m-1)}-y^{(m-1)}|^{m-1-\alpha}}\,dy^{(m-1)} \qquad   m=2,3~.
\end{eqnarray*}
Note that by Rademacher's theorem $\nabla f(y^{(m-1)})$ is well defined and is bounded at all points in $y^{(m-1)}\in B^{(m-1)}_\rho(0)$ except for a subset of Lebesgue measure zero.  Hence ${\mathbb I}_{1,2}$ are also bounded uniformly in $\eta\in [0, r]$ and $x\in \partial D\cap B_r(0)$. To  prove our lemma it thus suffices to estimate the  term ${\mathbb I}_3$, with the symmetric jumps. We provide the details of this estimation for $x=0$. For $x$ near $0$ (i.e., in $\partial D \cap B_r(0)$) the same approach works with obvious modifications. Since
$$\frac{\partial \Phi(x,y)}{\partial x_j}=\frac{-(x_j-y_j)}{\omega_m|x-y|^m}, \qquad m=2,3, \quad j=1\cdots m, \quad \omega_2=2\pi, \, \omega_3=4\pi~,$$ 
the integrals we need to study take the form
$$\int\limits_{B^{(m-1)}_\rho(0)}\left[ \frac{y_j}{\left(|y^{(m-1)}|^2+(f(y^{(m-1)})- \eta)^2\right)^{m/2}}-\frac{y_j}{\left(|y^{(m-1)}|^2+(f(y^{(m-1)})+\eta)^2\right)^{m/2}}\right] F(y^{(m-1)})dy^{(m-1)}$$ 
 for $j=1, \cdots, (m-1)$  and 
$$\int\limits_{B^{(m-1)}_\rho(0)}\left[\frac{(f(y^{(m-1)})- \eta)}{\left(|y^{(m-1)}|^2+(f(y^{(m-1)})- \eta)^2\right)^{m/2}}-\frac{(f(y^{(m-1)})+\eta)}{\left(|y^{(m-1)}|^2+(f(y^{(m-1)})+\eta)^2\right)^{m/2}}\right]F(y^{(m-1)})dy^{(m-1)},$$
for $j=m$. Here 
$$F(y^{(m-1)}):=\sqrt{1+|\nabla f(y^{(m-1))}|^2}\, \psi(y^{(m-1)}, f(y^{(m-1)}))\, \nu_i(y^{(m-1)}, f(y^{(m-1)})$$
is a function in $L^\infty(B^{(m-1)}_\rho(0))$, and hence there is a $C>0$ such that $|F(y^{(m-1)})|\leq C$ for almost all $y^{(m-1)}\in B^{(m-1)}_\rho(0)$. In order to estimate the above integrals, it therefore suffices to estimate  
\begin{equation}\label{ji}
\hskip -5pt
\int\limits_{B^{(m-1)}_\rho(0)}\left| \frac{y_j}{\left(|y^{(m-1)}|^2+(f(y^{(m-1)})- \eta)^2\right)^{m/2}}-\frac{y_j}{\left(|y^{(m-1)}|^2+(f(y^{(m-1)})+\eta)^2\right)^{m/2}}\right| dy^{(m-1)}
\end{equation}
for $j=1, \cdots, (m-1)$, and 
\begin{equation}\label{jm}
\hskip -5pt
\int\limits_{B^{(m-1)}_\rho(0)} \left| \frac{(f(y^{(m-1)})- \eta)}{\left(|y^{(m-1)}|^2+(f(y^{(m-1)})- \eta)^2\right)^{m/2}}-\frac{(f(y^{(m-1)})+\eta)}{\left(|y^{(m-1)}|^2+(f(y^{(m-1)})+\eta)^2\right)^{m/2}} \right| dy^{(m-1)}
\end{equation}
In fact these are exactly the integrands estimated in \cite[page 363-364]{pom3} using simple algebraic manipulations, which we have  included in Appendix \ref{ap1} for the reader's convenience. Upon substitution of $y^{(m-1)} = \eta u^{(m-1)}$ these calculations imply that the integrals (\ref{ji}) and (\ref{jm})  are bounded  by
$$\int\limits_{B^{(m-1)}_{\rho/\eta}(0)}\frac{1}{\left(|u^{(m-1)}|^2+1\right)^{m/2}}\, du^{(m-1)}<+\infty, \qquad , $$
uniformly in $0 <\eta\leq r$  and  $x\in \partial D\cap B_r(0)$. This completes the proof of Lemma \ref{c2}.
\end{proof}

\section{Proof of our main results}
\label{proof}
In our proof of the main results we shall make use of a regularity result about $H^2_0(D)$ solutions to (\ref{nsp}). A central ingredient in the proof of this regularity result is the regularity analysis for the volume potential found in the previous section.

\begin{proposition}\label{regs}
Assume that $\partial D$ is  Lipschitz,  $0\in \partial D$, and the refractive index is given by  $n\in L^\infty(D)$. Furthermore, we assume that  $n \in C^{\alpha}(\overline D \cap B_R(0))$ for some ball $B_R(0)$ of radius $R$ centered at $0$ and some $0<\alpha\leq 1$. Then  $u\in H^2_0(D)$, that satisfies (\ref{nsp}), lies in $C^1(\overline{D})$, and has all its second derivatives $\left\{u_{i,j}\right\}_{i,j=1,m}$ uniformly bounded in $D\cap B_r(0)$ for some $r>0$.
\end{proposition}
\begin{proof}
First we remark that the incident field $v$ is real analytic in $\overline D$, as it is an ($L^2$) solution of Helmholtz equation in a region containing $\overline D$. Introduce the function 
$$
U(x) = \begin{cases} u(x) &\hbox{ for } x \in D~, \\ 
0 &\hbox{ for } x \in \mathbb{R}^m\setminus D~.\end{cases}
$$ 
This function is in $H^2(\mathbb{R}^m)$ (since $u\in H^2_0(D)$) and since $m=2$ or $3$, it follows from the Sobolev embedding Theorem that $U \in C^{\alpha}(\mathbb{R}^m)$ for some $0<\alpha<1$.  $U$ is a solution of 
$$
\Delta U =\Psi \hbox{ in } \mathbb{R}^m ~~, ~~ \hbox{ where }  \Psi(x) =\begin{cases} \psi(x)&\hbox{ for } x \in D~ \\ 
0 &\hbox{ for } x \in \mathbb{R}^m\setminus D  \end{cases} 
$$
with $\psi(x)= k^2(1-n(x))v(x)-k^2n(x)u(x),~x \in D$. The function $\psi$ is clearly in $L^\infty(D)$, and due to the assumptions about $n$ and $v$, and the $C^\alpha$ extendability of $u$, it has an extension that lies in $C^\alpha(B_r(0))$. The solution $U$ is now given by the formula
$$
U(x)= - \int_D \psi(y) \Phi(x,y)\,dy =-w_\psi(x)
$$
with $\psi= k^2(1-n)v-k^2nu \in L^\infty(D)\cap C^{\alpha}(B_r(0))$. Lemma \ref{c1} and Lemma \ref{c2} of the preceding section therefore apply to $U$. Lemma \ref{c1} implies that $U\in C^1(\mathbb{R}^m)$ and, since $U=0$ outside $D$, Lemma \ref{c2} implies that all second derivatives of $u$ are uniformly bounded in $D \cap B_r(0)$ for some $r>0$. 
\end{proof}

\vskip 10pt
\begin{remark} \label{rem2}
In the above proof of Proposition \ref{regs} it is shown that $U$ is in $C^1(\mathbb{R}^m)$; as a consequence $u$ has an extension (by zero) which is in  $C^1(\mathbb{R}^m)$.
We also note that, the fact that all second derivatives of $u$ are shown to be uniformly bounded in $D\cap B_r(0)$ implies that $u$ is in $C^{1,1}(\overline{D}\cap B_r(0))$.
\end{remark}

\noindent
To obtain the main results of our paper we need to use Theorem \ref{th1}, which requires a real valued solution. With this in mind, we note that the real valued function $w =\Re(u)$ is an $H^2(D)$ solution to 
\begin{equation}\label{real1}
\Delta w+k^2nw=-k^2(n-1)\Re(v)~~\hbox{ with } w=\frac{\partial w}{\partial \nu}=0 \hbox{ on } \partial D~.
\end{equation}
Since the incident wave is an $L^2$ solution to $\Delta v+k^2v=0$ in a neighborhood of $\overline D$, it follows that $\Re(v)$ is a real analytic functions on $\overline D$. In particular, Proposition \ref{regs} also applies to $w$. Of course, one could consider the imaginary part of the scattered field $u$, which satisfies the same equation as above with $\Re(v)$ replaced by $\Im(v)$. Accordingly, in what follows, everything holds true if we replace $\Re(v)$ by $\Im(v)$. 

\noindent
To apply Theorem \ref{th1}  to  (\ref{real1}), we must first appeal to Theorem \ref{th2}  in order to establish that  $w\in C^2(\overline{D}\cap B_r(0))$ and that $\partial D\cap B_r(0)$ is of class $C^1$. Proposition \ref{regs} (see also Remark \ref{rem2}) guaranties that $w\in C^{1,1}(\overline D\cap B_r(0))$ and that $g= -k^2(nw+(n-1)\Re(v))$ has a $C^1$ extension to all of $\mathbb{R}^m$. The essential, missing step for application of Theorem \ref{th2} is therefore to show that $w$ is of one sign. This is established by the following proposition.

\begin{proposition}\label{prep2}
Assume that $\partial D$ is  Lipschitz,  $0\in \partial D$, and the refractive index is given by  $n\in L^\infty(D)$. Furthermore, suppose $n$ lies in $C^{1,1}(\overline{D}\cap B_r(0))$ for some ball $B_r(0)$ of radius $r$ centered at $0$, and suppose $(n(0)-1)\Re (v(0))\neq 0$.  Let  $w\in H^2_0(D)$ be a solution to (\ref{real1}). Then $w<0$ in $D\cap B_r(0)$ for some $r>0$ if $(n(0)-1)\Re (v(0))>0$, and $w>0$ in $D\cap B_r(0)$ for some $r>0$ if $(n(0)-1)\Re (v(0))<0$.
\end{proposition} 
\begin{proof}
The proof of this proposition follows almost verbatim the analysis by Williams in \cite[Section 5]{pom3}. However, since \cite{pom3} deals with a slightly simpler equation, and since we assume these techniques may not be known to the reader, we provide the main steps of the proof. We provide sufficient details where our case differs from the one considered in  \cite[Theorem 3]{pom3}, and otherwise refer the reader to \cite{pom3}.

\noindent
To fix ideas we consider the case $(n(0)-1)\Re(v(0))>0$. In the case when  $(n(0)-1)\Re(v(0))<0$ the result is verified by considering $-w$ (and $-v$) instead of $w$ (and $v$). Since  $(n-1)\Re(v)$ is  $C^1( B_r(0))$, by decreasing $r$ if necessary, it now follows that
\begin{equation}\label{fsign}
\mbox{$(n(x)-1)\Re(v(x))\geq \gamma >0$  for all  $x\in \overline{D}\cap B_r(0)$.}
\end{equation}

\noindent
We consider the cylindrical neighborhood $N$ of $0\in \partial D$ introduced  in Section \ref{secreg} in the paragraph just before Lemma \ref{c2}. We recall  that  $N:=N(\rho,h)=B_\rho^{(m-1)}(0)\times [-h, h]$, $N\cap \partial D$ is the graph $x_m=f(x^{(m-1)})$ of a Lipshitz continuous function $f:B_\rho^{(m-1)}(0)\to {\mathbb R}$ with Lipschitz constant $K$, $h>K\rho$,  and 
$$
N\cap D=\left\{(x^{(m-1)}, x_m)~:~ \, x^{(m-1)}\in B_\rho^{(m-1)}(0), \, f(x^{(m-1)})<x_m<h\right\}.
$$
The function $w\in H_0^2(D)$, which in Proposition \ref{regs} is shown to have all its second derivatives $\{w_{i,j}\}_{i,j=1,m}$  uniformly bounded in $N\cap D$, solves 
\begin{equation}\label{ewij}
\Delta w=-k^2(nw+(n-1)\Re(v)) \qquad \mbox{in }\; N\cap D.
\end{equation}
Set $g:=-k^2(nw+(n-1)\Re(v))$. From our assumption about $n$ and the analyticity of $\Re (v)$ we have that $g_{i,j}:=\displaystyle{\frac{\partial^2 g}{\partial x_i\partial x_j}}$, $i,j=1..m$, exist almost everywhere in  $N\cap D$ and are in $L^\infty(N\cap D)$. Furthermore,
$$\Delta w_{i,j}=g_{i,j}\qquad \mbox{almost everywhere in}\; N\cap D.$$
Using Lemma \ref{c1} we can (by means of a volume potential) construct  functions $q_{i,j}\in C^1({\mathbb R}^m)$ such that
$$ 
\Delta q_{i,j}=g_{i,j}\qquad \mbox{almost everywhere in} \; N\cap D.
$$
Thus each $w_{i,j}-q_{i,j}$ is a bounded harmonic function  in $N\cap D$. By taking the radius $\rho$ of the ball $B_\rho^{(m-1)}(0)$ sufficiently small we can assume that $N\cap D$ is starlike about some point in $N\cap D$. Therefore, from \cite[Section 2, page 311]{hw} we can conclude that each $w_{i,j}-q_{i,j}$ has non-tangential limits at all points of $N\cap \partial D$,  except for a possible (Borel) subset of  zero harmonic measure. Since $q_{i,j}\in C^1({\mathbb R}^m)$, we can then conclude that each $w_{i,j}$ has finite non-tangential limits on $N\cap \partial D$ except for a possible subset of  zero harmonic measure.  From geometric measure theory it is known (see e.g. \cite[Theorem 1, page 275]{dah} or \cite{harmonic}) that the (Borel) boundary sets, of a Lipschitz domain, which have harmonic measure zero, are exactly those boundary sets which have $(m-1)$-dimensional  Hausdorff measure zero. Therefore each $w_{i,j}$ has  finite non-tangential limit  at $\left(x^{(m-1)}, f(x^{(m-1)})\right)$ for all  $x^{(m-1)}\in B_\rho^{(m-1)}(0)$, except for a possible subset of $(m-1)$-dimensional zero Lebesgue measure. 

\noindent
By Rademacher's theorem $f:B_\rho^{(m-1)}(0)\to {\mathbb R}$ is differentiable almost everywhere  in $B_\rho^{(m-1)}(0)$. We  now introduce  the subset
$$G:=\left\{x^{(m-1)}\in B_\rho^{(m-1)}(0), \,\mbox{where both  $\nabla f$ and the non-tangential limits of all $w_{i,j}$ exist}\right\}.$$ Note that $B_\rho^{(m-1)}(0)\setminus G$ has zero Lebesgue measure. We note that the non-tangential limit of $w_{i,j}$ at a  $x_0:=\left(x_0^{(m-1)}, f(x_0^{(m-1)})\right)$ (for $x_0^{(m-1)} \in G$) is the limiting value as we approach $x_0$  by  $x\in N\cap D$  from inside any cone $C_\epsilon(x_0):=\left\{x: (x-x_0)\cdot \nu_{x_0}\leq -\epsilon |x-x_0| \right\}$, $\epsilon>0$,  where $\nu_{x_0}$ denotes the outward normal vector to $\partial D$ at $x_0$.

\noindent
Next we compute the non-tangential limits of $w_{i,j}$ at  $\left(x^{(m-1)}, f(x^{(m-1)})\right)$ for $x^{(m-1)}\in G$. For a fixed $x_0^{(m-1)}\in G$ we can  setup a local coordinative system such that  $\left(x_0^{(m-1)}, f(x_0^{(m-1)})\right)=(0^{(m-1)},0)$,  $x_m=0$ coincides with the tangential plane to the graph of $f$, and the points $(0^{(m-1)},h)$ for $h>0$ small enough are in $D$. $\partial_i$, $1\le i\le m-1$ with respect to this local coordinate system then denotes a tangential derivative to $\partial D$. Following \cite[Lemma 2.1(b') page 82]{caf2}, we consider the $(m-1)$ dimensional disk regions (inside $D$) with radius $\rho_\ell$, defined as $D_\ell:= t_\ell\cap C_\epsilon(0)$ where $\{t_\ell\}_{\ell\in {\mathbb N}}$ is a sequence of planes parallel to the tangential plane to $\partial D$ at $0$, and converging to it. Then we have
\begin{equation}
\label{average}
\frac{1}{\rho_l^{m-1}}\left|\int\limits_{D_\ell}w_{i,j} \, dx\right|\leq \frac{1}{\rho_l^{m-1}}\int\limits_{\partial D_\ell} \left| w_j \right|\, ds\leq C\epsilon ~, \qquad \mbox{as $\ell \to \infty$}~,
\end{equation}
because $w_j$ is Lipshitz continuous and vanishes on the free boundary (and the distance to the boundary is $\rho_l \epsilon$).  Since $w_{i,j}$ has a limit, call it $l_{i,j}$, from within $C_\epsilon(0)$ we may conclude from (\ref{average}) that $| l_{i,j} |<C\epsilon$ which implies that $\l_{i,j}=0$ since $\epsilon>0$ is arbitrary. There is only one remaining second derivative, namely the one corresponding to differentiation twice with respect to the $m'th$ local variable, whose limit we need to calculate. This second derivative actually coincides with $\frac{\partial^2}{\partial \nu_0^2}w$, where $\nu_0$ denotes the outward normal to $\partial D$ (at $(x_0^{(m-1)}, f(x_0^{(m-1)}))$). Using (\ref{ewij}), the fact that $w$ vanishes on the boundary and  is uniformly continous up to the boundary, together with (\ref{fsign}) we may now conclude that limit of  $w_{\nu_0,\nu_0}$ is $l_{\nu_0,\nu_0}\leq- k^2 \gamma<0$, and this holds for all $x_0^{(m-1)}\in G$.

\noindent
Now  we go  back to the fixed global coordinate system with the fixed $0\in \partial D$, and we denote by $\nu(x^{(m-1)})$ the (outward) normal vector to the tangent plane to the graph of $f$ at $\left(x^{(m-1)}, f(x^{(m-1)})\right)$ for  $x^{(m-1)}\in G$. Since $f$ is Lipshitz we have that the angle between $-\nu(x^{(m-1)})$ and $e_m$ (the unit vector in the  $m$-direction) satisfies  $\theta(-\nu(x^{(m-1)}), e_m)\leq \theta_0<\pi/2$ for all  $x^{(m-1)}\in G$.  Then from the above we obtain that $w_{m,m}$ (the non-tangential limit of the second derivative in the $e_m$ direction)  satisfies 
\begin{equation}\label{ggg}
w_{m,m}(x^{(m-1)}, f(x^{(m-1)})\leq -k^2 \gamma \cos^2(\theta_0)=-\epsilon<0 \; \qquad \mbox{for all $x^{(m-1)}\in G$}~.
\end{equation}
The idea is now, based on the sign property (\ref{ggg}), to construct a subharmonic function involving $w_{m,m}$ that takes negative values uniformly on the boundary of a neighborhood of $0\in \partial D$ in $\overline{D}$, and then use a maximum principle for subharmonic functions to infer the same sign property inside the neighborhood.

\noindent
To this end, let $P>0$ be a positive constant. Then from (\ref{ewij}) we have
$$\Delta\left(w_{m,m}(x)+P |x |^2\right)=\Delta w_{m,m}(x)+2mP=g_{m,m}(x) +2mP\qquad \mbox{for}\; x\in N\cap D~.$$
Since  $g_{m,m}$ is uniformly bounded in $x\in N\cap D$,  it is possible to choose $P>0$ large enough so that $g_{m,m}(x) +2mP\geq 0$ for $x\in N\cap D$. Thus the function  $w_{m,m}+P|x|^2$ is subharmonic in $N\cap D$. Let $K$ be the Lipschitz constant of $f$. Now, pick a smaller neighborhood of $0$, $N^*:=N(\rho^*,h^*)$ with $\rho^*<\rho$, $h^*<h$, $h^*>K \rho*$ so that $P|x|^2\leq \epsilon/2$ in $N^*$ where $\epsilon$ is the constant in (\ref{ggg}).  The boundary of $N^*\cap D$ can be split into  $\partial (N^*\cap D)=\Gamma_1\cup \Gamma_2$ where 
$$\Gamma_1:=\left\{(x^{(m-1)}, x_m)~:~ \, x^{(m-1)}\in \overline{B_{\rho^*}^{(m-1)}(0)}, \, f(x^{(m-1)})<x_m\right\}\cap \partial(N^*\cap D)~.
$$
and
$$\Gamma_2:=\left\{(x^{(m-1)}, x_m)~:~ \, x^{(m-1)}\in \overline{B_{\rho^*}^{(m-1)}(0)}, \, f(x^{(m-1)})=x_m\right\}~.
$$
In particular we have that $w_{m,m}+P|x|^2\leq - \epsilon/2$ almost everywhere  on $\Gamma_2$ (in the sense of non-tangential limits).  The next step is to control the sign on $\Gamma_1$. For this purpose, we consider the (exterior) cone ${\mathcal C}$ with vertex $0$ and opening $\eta$ with $0<\eta<\tan^{-1}(1/K)$
$${\mathcal C}:=\left\{x:=\left(x^{(m-1)},x_m\right)\in {\mathbb R}^m: \, x_m<0,\; |x^{(m-1)}|<|x_m | \tan \eta\right\}~.$$
Note that $\overline D\subseteq {\mathbb R}^m\setminus {\mathcal C}$.
It is quite easy (see e.g. \cite[Lemma 2.4, page 62]{HK}) to construct a barrier function $u$ defined and continuous in ${\mathbb R}^m\setminus {\mathcal C}$, subharmonic in ${\mathbb R}^m\setminus \overline{\mathcal C}$, and such that $u(x)\leq 0$ for $x\in {\mathbb R}^m\setminus {\mathcal C}$, with equality holding if and only if $x=0$. Now consider the subharmonic function $w_{m,m}+P |x|^2+Mu$ for some constant $M>0$, to be chosen later.  Since $Mu\leq 0$ we still have
\begin{equation}\label{bdd1}
w_{m,m}(x)+P |x |^2+Mu(x)\leq -\epsilon/2 \;\qquad \mbox{ for almost all} \, \; x\in \Gamma_2~.
\end{equation}
Since $\Gamma_1$ is away from $0$, we have that $\sup_{x\in \Gamma_1}u(x)<0$. Therefore, thanks to the fact that  $w_{m,m}$  is uniformly bounded in $N\cap \overline D$,  we can find a constant $M>0$ such that 
\begin{equation} \label{bdd2}
w_{m,m}(x)+P |x|^2+Mu(x)\leq -\epsilon/2 \;\qquad \mbox{ for} \,\; x\in \Gamma_1~.
\end{equation}
Lemma \ref{cl1} (see below) now implies  that the above inequality holds in the interior of $N^*\cap D$, in particular we may conclude
$$ 
w_{m,m}(x)\leq -\epsilon/2 - P |x|^2- Mu(x)\;\qquad \mbox{ for all} \, \; N^*\cap D~.
$$
Since $u(x)\to 0$ as $|x|\to 0$ for $x\in  {\mathbb R}^m\setminus {\mathcal C}$, we can therefore find a sufficiently small  $\rho_0>0$ such that 
$$ w_{m,m}(x)\leq -\epsilon/4 \;\; \mbox{for all $x\in N^*\cap D$ with $|x|\leq \rho_0$}.$$
This, along with the fact that  $w=\partial w/\partial x_m=0$ on $\Gamma_2$, implies that $w<0$ along all lines in the direction $e_m$ (inside $N^* \cap D \cap \{ |x|<\rho_0\}$) which finally proves  that $w<0$ in a neighborhood of $0$ in $D$.\end{proof}
\begin{lemma}\label{cl1}
Let the notations and assumptions be as in the proof of Proposition \ref{prep2}. If (\ref{bdd1}) and (\ref{bdd2}) hold then
$$w_{m,m}(x)+P|x|^2+M u(x)\leq -\epsilon/2 \qquad \mbox{for all $x\in N^*\cap D$~.}$$ 
\end{lemma}
\noindent
This result (though not directly stated as a lemma) is proven from bottom of page 366 through the top of page 368 in Williams \cite{pom3}, and we refer the reader to that paper. This lemma amounts to a ``maximum principle" for the subharmonic function  $w_{m,m}(x)+P|x|^2+M u(x)$, the added difficulty being the lack of apriori knowledge that this function is continuous on $\overline{N^*\cap D}$. 

\bigskip

\noindent
We are now ready to give the proof of the first main result of this paper.
\begin{proof}[Proof of Theorem \ref{scatnon1}]
The proof proceeds by contradiction. Suppose the incident field $v$ is not scattered by $(D,n)$.
Without loss of generality we assume that $(n(x_0)-1)\Re(v(x_0)) \ne 0$ and we choose $x_0$ to be the origin of the coordinate system (the argument works similarly if $(n(x_0)-1)\Im(v(x_0)) \ne 0$). The function $w =\Re(u)$ is a solution to (\ref{real1}). Since the incident wave is an $L^2$ solution to $\Delta v+k^2v=0$ in a neighborhood of $\overline D$, it follows that $\Re (v)$ is a real analytic functions on $\overline D$.  
By assumption the refractive index $n$ is also real analytic on $\overline{D_\delta}$, and so the assumptions of Proposition \ref{regs}  and Proposition \ref{prep2} are satisfied.  In particular, Proposition \ref{regs} (and the remark following) implies that $w\in C^{1,1}(\overline D\cap B_r(0))$ for some ball $B_r(0)$, and that it has a $C^1$ extension to all of $\mathbb{R}^m$. Proposition \ref{prep2} implies that $w\geq 0$ or $w\leq 0$ in $\overline D\cap B_r(0)$ depending on whether $(n(0)-1)\Re(v(0))<0$ or $(n(0)-1)\Re(v(0))>0 $ respectively. We now introduce $g:=-k^2(nw+(n-1)\Re(v))$. Thanks to the $C^1$ extendability of $w$, and the analyticity of $n$ and $\Re(v)$, the function $g$ has  a $C^1$-extension $g^*$ in a neighborhood of  $\overline{D}\cap B_R(0)$. Since $w$ vanishes at $\partial D$, $g(0)=-k^2(n(0)-1)\Re(v(0))$, and so it follows that $g^*\geq \gamma>0$  in $\overline{D}\cap B_r(0)$ or $g^*\leq -\gamma<0$  in $\overline{D}\cap B_r(0)$ (with $r$ sufficiently small) depending on whether $(n(0)-1)\Re(v(0))<0$ or $(n(0)-1)\Re(v(0))>0$, respectively. Since $w$ satisfies $\Delta w= g$ in $D$, the assumptions of Theorem \ref{th2} are now satisfied for $w$, if $(n(0)-1)\Re(v(0))>0 $, and for $-w$ if $(n(0)-1)\Re(v(0))<0 $. In both cases we may therefore conclude that $w \in C^2(\overline {D}\cap B_r(0))$  and $\partial D\cap B_r(0)$ is of class $C^1$.

\noindent
We now apply Theorem \ref{th1}. We set $a(x) = k^2n(x)$ and $b(x)=k^2(1-n(x))\Re (v(x))$, then $a$ and $b$ are both real analytic, by assumption $a(0), \,b(0)\ne 0$ and $w \in C^2(\overline{D}\cap B_r(0))$ satisfies
$$
\Delta w +a w = b \hbox{ in } D~,  \hbox{ with } w=\frac{\partial w}{\partial \nu}=0 \hbox{ on } \partial D ~,
$$
where $\partial D\cap B_r(0)$ is known to be of class $C^1$. The third case in Theorem \ref{th1} yields that $\partial D \cap B_r(0)$ is real analytic for $r$ sufficiently small. However, this represents a contradiction, and so we conclude that the incident field $v$ is scattered by $(D,n)$, thus completing the proof of Theorem \ref{scatnon1}.

\end{proof}
\noindent
Our second main result, Theorem \ref{scatnon2}, which applies to less regular refractive index $n$ is proven in the exact same manner. The regularity result of  Proposition \ref{regs} and  Proposition \ref{prep2} are still applicable, since we have assumed that $n\in C^{1,1}(\overline D_\delta)$. For the free boundary regularity we rely on the case 2 of Theorem \ref{th1}. 

\vskip 10pt
\noindent
We close this section with a remark on up-to-the-boundary regularity of the $v$-part of the transmission eigenfunction. At a real transmission eigenvalue $k>0$,  there exist nonzero $u\in H^2_0(D)$ and $v\in L^2(D)$, which solve
\begin{eqnarray}
&\Delta u+k^2nu=k^2(1-n) v &  \; \mbox {in}\; D \label{t1}\\
&\Delta v+k^2v=0 &  \; \mbox {in}\; D \label{t2}
\end{eqnarray}
Without loss of generality, we may assume that the eigenfunction $(u,v)$ is real valued. In general, since $v$  assumes no boundary condition, it is not possible from the equations to conclude any regularity for $v$ up to the boundary.  Our free boundary regularity results provide some insight into this issue.  Recall that  Theorem \ref{scatnon1}  and Theorem   \ref{scatnon2}  state necessary regularity conditions on $\partial D$, in order that (\ref{t1}) can have a $H_0^2(D)$ solution ($v$ being defined and regular in a $\mathbb{R}^m$ neighborhood of $\partial D$).  It is clear from our analysis that the statements of Theorem \ref{scatnon1}  and Theorem   \ref{scatnon2}   are valid if $v$ is only defined on one side of $\partial D$, and  the regularity of $v$ up to the boundary matches that of $n$; simply notice that  our arguments rely only on the local regularity of the source term $(1-n)v$ in ${\overline D}\cap B_R(x_0)$. We thus have  the following cosequence of the proofs of Theorem \ref{scatnon1}  and Theorem \ref{scatnon2}.
\noindent
\begin{corollary}\label{cor1}
Assume $k>0$ is a real transmission eigenvalue with eigenfunction $(u, v)$,  $\partial D$ is Lipshitz, $n\in L^\infty(D)$,  and there exits $x_0\in \partial D$ such that $n(x_0)-1\neq 0$. The following assertions hold:
\begin{enumerate}
\item  If $n$ is real analytic in a neighborhood of $x_0$  and $\partial D\cap B_r(x_0)$ is not real analytic for any ball $B_r(x_0)$, then $v$ can not be real analytic  in any neighborhood of $x_0$, unless $v(x_0)=0$.
\item If $n\in C^{m, \mu}(\overline{D}\cap B_R(x_0))\cap C^{1,1}(\overline{D}\cap B_R(x_0))$ for $m \geq 1$, $0<\mu<1$ and some ball $B_R(x_0)$, and  $\partial D\cap B_r(x_0)$  is not of class $C^{m+1, \mu}$ for any ball $B_r(x_0)$, then $v$ cannot lie  in $C^{m, \mu}(\overline{D}\cap B_r(x_0))\cap C^{1,1}(\overline{D}\cap B_r(x_0))$ for any ball $B_r(x_0)$, unless $v(x_0)=0$.
\end{enumerate}
\end{corollary}

\section{Applications to special incident waves}\label{appl}
\label{specinc}
In this section we describe some applications of our main results to broad classes of incident waves. To illustrate our results from a different perspective we shall formulate these applications in terms of the boundary regularity of $D$ implied by a lack of scattering.
For an incident plane wave
$$
v_\xi(x)= e^{ik\xi\cdot x}~,
$$
and an index of refraction $n$ with the property that $n(x) \ne 1$ for all $x\in \partial D$, one has
$$
k^2(1-n(x))v_\xi(x) \ne 0~, ~~x \in \partial D~, ~~ k \ne 0~,
$$
and so the non-degeneracy condition of our main results is satisfied for all $x\in \partial D$. As a consequence
we have the following corollary to Theorem \ref{scatnon1}
\begin{corollary} \label{planew}
Suppose $\partial D$ is Lipshitz. Suppose the index of refraction $n\in L^\infty(D)$, $n(x)\ge n_0>0$, is real analytic in $\overline{D_\delta}$, with $n(x)\ne 1$ for all $x\in \partial D$. 
Let $u\in H^2_{loc}(\R^m)$ denote the solution to the problem (\ref{total})-(\ref{somer}), given the incident plane wave $v_\xi(x)= e^{ik\xi\cdot x}$, $\xi \in S^{m-1}$. If $k> 0$, and if $u$ vanishes identically in $\R^m\setminus \overline{D}$, then the boundary $\partial D$ is real analytic.
\end{corollary}
\begin{remark}
{\em We have two remarks related to  Corollary \ref{planew}:
\begin{enumerate}
\item Let $\Phi_k$ denote the fundamental solution to the Helmholtz equation give by
$$
\Phi_k(x,y):=\left\{\begin{array}{rrcll}\displaystyle{\frac{e^{ik|x-y|}}{4\pi |x-y|}} \quad \; & \qquad  \mbox{in }\, {\mathbb R}^3 \\
& \\
\displaystyle{\frac{i}{4}}\displaystyle{H^{(1)}_0(k|x-y|)} & \qquad \mbox{in }\, {\mathbb R}^2 ~.
\end{array}\right.~,
$$
where $H^{(1)}_0$ denotes the Hankel function of the first kind of order zero.
The result in Corollary \ref{planew} can be stated verbatim if the incident field $v$ is a point source, i.e.,
$$v_{z_0}(x):=\Phi_k(x,z_0) \qquad z_0 \in \mathbb{R}^m\setminus \overline{D}~ ,$$
since $v_{z_0}(x)\neq 0$ for all $x\in \partial D$.
\item If the refractive index $n\in C^{m, \mu}(\overline D_\delta)\cap C^{1,1}(\overline D_\delta) $, for $m\geq 1$, $0<\mu<1$ and  under the same additional assumptions  one can conclude that if $u$ vanishes identically in $\R^m\setminus \overline{D}$, then the boundary $\partial D$ is of class $C^{m+1, \mu}$.
\end{enumerate}
}
\end{remark}
\noindent 
A natural question arises concerning the possible appearance of non-scattering for plane waves and obstacles with real analytic boundaries as well as real analytic index of refraction. A result in that direction is found in \cite{vx}. In that paper it is shown that if $D\subset \R^2$ is strictly convex (positive curvature) and with constant index of refraction, then given any direction $\xi$ there exists at most finitely many positive wave-numbers $k$ for which the plane wave in the direction $\xi$ does not scatter. If $D$ is a disk it is quite easy to see that a plane wave will scatter at any positive wave-number.

\vskip 10pt
\noindent
For the next application we consider the  two dimensional case, and incident waves obtained by superposition of plane waves, so-called Herglotz wave functions, of the form
\begin{equation}\label{Hergdef}
v_\phi(x)= \frac{1}{2\pi}\int_{S^1} \phi(\xi) e^{ik\xi\cdot x}~ds_\xi~,
\end{equation}
where we take $\phi$ to be a $C^1$ function. While the free boundary regularity result in Theorem \ref{scatnon1} does not insure that non-scattering for such incident waves can only occur for obstacles with real analytic boundaries, it does imply that (for real analytic internal index of refraction) infinitely many non-scattering positive wave-numbers can only occur if the boundary of the obstacle is real analytic, except possibly at a nowhere dense (rare) set of points. To make this statement precise we introduce the real analytic ``part" of the boundary
$$
\partial D_A = \{ x\in \partial D~:~ \partial D \hbox{ is real analytic in a neighborhood of } x \}~~.
$$

\begin{corollary}\label{Hwave}
Suppose $\partial D$, the boundary of the domain $D\subset \mathbb{R}^2$, is Lipschitz. Suppose the index of refraction $n\in L^\infty(D)$, $n(x)\ge n_0>0$, is real analytic on $\overline{D_\delta}$, with $n(x)\ne 1$ for all $x\in \partial D$. For fixed $K_0$ and $\epsilon_0>0$, let $\Phi$ denote the set
$$
\Phi = \left\{  \phi \in C^1(S^1)~~,~ \Vert \phi \Vert_{C^1} \le K_0 ~~,~ \hbox{and } |\phi|>\epsilon_0 \hbox{ on } S^1~\right\}~,
$$ 
and let $u_\phi \in H^2_{loc}(\R^m)$ denote the solution to the scattering problem (\ref{total})-(\ref{somer}), given an incident Herglotz wave  function $v_\phi$ of the form (\ref{Hergdef}). Suppose there exists an infinite sequence of positive wave-numbers $k_j$ and associated scattering solutions $u_{\phi_j}$, with $\phi_j\in \Phi$, for which 
$$
u_{\phi_j} \hbox{ vanishes identically in } \R^2\setminus \overline{D}~,
$$
then $\partial D_A$ is a dense open set, or equivalently: the complement of $\partial D_A$ is a closed nowhere dense set.
\end{corollary}
\begin{proof}
The set $\partial D_A$ is clearly an open subset of $\partial D$.
Let $x^*$ be an arbitrary point on $\partial D$, we proceed to show any neighborhood of $x^*$ contains a point from $\partial D_A$. For this purpose we may without loss of generality assume $x^*\ne0$ (since a set which is dense in $\partial D$ minus a point is also dense in $\partial D$).  We analyze two exhaustive, but mutually exclusive possibilities
\begin{itemize}
\item The map $x \rightarrow |x|$ is constant in a neighborhood of $x^*$.
\vskip 5pt
\item For any $\partial D$ neighborhood $\omega$ of $x^*$, the image of $\omega$ under the map $x \rightarrow |x|$ contains an open interval. 
\end{itemize}
In the first case the boundary $\partial D$ is part of a circle near $x^*$, therefore locally real analytic, and so $x^*$ itself lies in $\partial D_A$. In order to deduce that any neighborhood of  $x^*$ also contains a point from $\partial D_A$ in the second case, we shall apply the result from Theorem \ref{scatnon1}. Now suppose any neighborhood of $x^*$ contains a point $z$ with 
$$
v_{\phi_j}(z)=\frac{1}{2\pi}\int_{S^1} \phi_{j}(\xi) e^{ik_{j}\xi\cdot z}~ds_\xi \ne 0
$$ 
for some $j$. Since $v_{\phi_j}$ is not scattered, it follows from Theorem \ref{scatnon1} that $\partial D$ is analytic near the point $z$, and so it follows that any neighborhood of $x^*$ contains a point from $\partial D_A$. We are thus left to consider the second case when we also know that there exists a neighborhood $\omega$ of $x^*$  such that 
\begin{equation}
\label{oscint}
\frac{1}{2\pi}\int_{S^1} \phi_{j}(\xi) e^{ik_{j}\xi\cdot z}~ds_\xi = 0~, ~~\hbox{ for all } j \hbox{ and for all }z \in \omega~.
\end{equation}
By decreasing $\omega$, if necessary, we may assume $|z|>c>0$ in $\omega$.
In the following we show, by contradiction,  that this situation is vacuous, and having done so, we may conclude that any neighborhood of $x^*$ contains a point from $\partial D_A$ also in case two. This will complete the proof that $\partial D_A$ is dense in $\partial D$. Now to establish the contradiction:
note that wave numbers associated with non-scattering (nontrivial) incident waves are automatically transmission eigenvalues. Because $n(x)\ne 1$ on $\partial D$, and because of the regularity of $n$ near $\partial D$, it follows that either (1) $\min_{x\in D_\delta} n(x)>1$, or (2) $\max_{x\in D_\delta}n(x)<1$, for some $\delta$ sufficiently small.  Due to this "sign" condition on $n$ it is known that the only accumulation point of the transmission eigenvalues is at $\infty$, in other words, we know that $k_j \rightarrow \infty$ as $j \rightarrow \infty$. The stationary phase approximation to integrals such as that in (\ref{oscint}) asserts that, for any $\phi \in C^1$,
\begin{eqnarray*}
\frac{1}{2\pi}\int_{S^1} \phi(\xi) e^{ik_{j}\xi \cdot z}~ds_\xi &=& \frac{1}{2\pi}\int_{-\pi}^\pi \phi(\theta) e^{ik_{j}|z| cos (\theta-\theta_z)}~d\theta \\
&=&\frac1{2\pi}\phi(\theta_z)e^{ik_{j}|z|}e^{-i\pi/4}\left(\frac{2\pi}{k_j |z|}\right)^{1/2} \\
&&\hskip 20pt +\frac1{2\pi}\phi(\theta_z+\pi)e^{-ik_{j}|z|}e^{i\pi/4}\left(\frac{2\pi}{k_j|z|}\right)^{1/2} +o(k_j^{-1/2})~,
\end{eqnarray*}
for any $z \in \partial D$, $z \ne 0$. Here we have parametrized $\xi\in S^1$ by angle $\theta \in (-\pi,\pi)$: $\xi = (\cos \theta,\sin \theta)$, and interpreted $\phi$ as a periodic function of $\theta$. We have also  written $z= |z| (\cos \theta_z,\sin \theta_z)$. Furthermore the remainder term $o(k_j^{-1/2})$ is uniform over $\Vert \phi \Vert_{C^1} \le K$, (and $|z|>c>0$). By insertion of $\phi=\phi_j \in \Phi$, use of (\ref{oscint}), and  rearrangement we now get
\begin{equation} \label{firstlim}
\phi_j(\theta_z)e^{2ik_{j}|z|}+i \phi_j(\theta_z+\pi) \rightarrow 0 \hbox{ as } j \rightarrow \infty~,
\end{equation}
for any $z \in \omega$. Since the $\phi_j$ are bounded in $C^1$, we may extract a subsequence (for simplicity, still indexed by $j$) that converges to some $\phi$ in $C^0$; this limit $\phi$ also satisfies $|\phi(\theta)|>\epsilon_0 $ for all $\theta$. From the limiting statement (\ref{firstlim}) we conclude that
$$
\phi(\theta_z)e^{2ik_{j}|z|}+i \phi(\theta_z+\pi) \rightarrow 0 \hbox{ as } j \rightarrow \infty~,
$$
for any $z \in \omega$. Since the image under the map $z \rightarrow |z|$ of any ($\partial D$) neighborhood of $x^*$ contains an open interval, Lemma 6.1 in \cite{pom1} asserts that there exists points $z_1$ and $z_2$ in $\omega$ so (after extraction of a subsequence)
$$
e^{2ik_j |z_1|} \rightarrow 1 ~\hbox{ and } ~ e^{2ik_j |z_2|} \rightarrow -1 ~.
$$
For the convenience of the reader we have included the statement of this lemma in Appendix \ref{ap2}. It thus follows that 
$$
\phi(\theta_{z_1})+i \phi(\theta_{z_1}+\pi) =0 ~\hbox{ and } -\phi(\theta_{z_2})+i \phi(\theta_{z_2}+\pi) =0~. 
$$
Since the above argument remains valid when we decrease the $x^*$ neighborhood $\omega$, we may achieve that both $z_1$ and $z_2$ are arbitrarily close to $x^*$. Therefore, by continuity
$$
\phi(\theta_{x^*})+i \phi(\theta_{x^*}+\pi) =0 ~\hbox{ and } -\phi(\theta_{x^*})+i \phi(\theta_{x^*}+\pi) =0~. 
$$
or
$$
\phi(\theta_{x^*})=\phi(\theta_{x^*}+\pi)=0~, 
$$
in contradiction to the fact that $|\phi(\theta)|$ is always positive.
\end{proof}

\noindent
For inhomogeneities  with real analytic boundaries some recent results about the number of positive non-scattering wave-numbers associated to incident Herglotz wave function (and constant index of refraction $\ne 1$) are found in \cite{vx}. For a circle there are infinitely many such wave-numbers associated to each density $\phi_j(\theta) = e^{\pm ij \theta}$. However, when the circle is perturbed (ever so slightly) to an ellipse, there can at most be finitely many such wave-numbers associated to any fixed $\phi$ (or any compact class of $\phi's$). The finiteness remains stable to perturbations of the ellipse. 

\medskip 
\noindent
We close this section with applications of our main results by establishing a scattering result for an inhomogeneous media $(D,n)$ at a wave number $k>0$, for which $k^2$ is not a Dirichlet eigenvalue of the negative Laplacian in $D\subset {\mathbb R}^m$. 
\begin{corollary}\label{dir}
Suppose $\partial D$ is Lipschitz. Suppose $k^2>0$  is not a Dirichlet eigenvalue of $-\Delta$ in $D$,  and that the index of refraction $n \in L^\infty(D)$ is real analytic on $\overline{D_\delta}$, with $n(x)\ne 1$ for all $x\in \partial D$. Furthermore, assume that $\partial D_A$ is the empty set. Then every incident wave  $v$ is scattered by this inhomogeneity.
\end{corollary}
\begin{proof}
$v$ is a (real-) analytic solution of the Helmhotz equation, $\Delta v+k^2v=0$, in a  region  containing ${\overline D}$. Since $k^2$ is not a Dirichlet eigenvalue, there is an open subset  ${\mathcal O}\subset \partial D$ of the boundary where $v$ does not vanish. In particular there exists a point $x_0\in {\mathcal O}\subset \partial D$ where the assumptions of Theorem \ref{scatnon1} are satisfied, and thus $v$ produces a non-zero scattered field.
\end{proof}

\noindent
If $k>0$ is not a transmission eigenvalue, we know that the inhomogeneity always scatters, hence the statement of Corollary \ref{dir} asserts that, non-scattering (with $\partial D_A =\emptyset$) can only occur for wave numbers $k>0$, that are transmission eigenvalues, and for which $k^2$ is a Dirichlet eigenvalue for $-\Delta$ in $D$. The cardinality of this set is not known.

\section{Remarks on non-radiating sources}
Our analysis has some implications for the scattering problem given a compactly supported source. More specifically, the scattered field due to  a given source $f\in L^2_c({\mathbb R}^m)$  of compact support  satisfies 
\begin{equation}\label{sourse}
\Delta u+k^2u=f \qquad \mbox{in}\; {\mathbb R}^m
\end{equation}
together with the outgoing Sommerfeld radiation condition (\ref{somer}). Again the outgoing scattered field $u$ exhibits the  following asymptotic behavior  as $r:=|x|\to \infty$
$$
u(x)=\frac{\exp(ikr)}{r^{\frac{m-1}{2}}}u^\infty(\hx)+O\left(r^{-\frac{m+1}{2}}\right)~,
$$
which defines the far field pattern $u^\infty(\hx)$ as a function on the unit $m-1$ sphere. There are plenty of compactly supported sources that produce zero far field patterns. For instance,  the set 
$$\left\{f:=\Delta v+k^2v, \quad \mbox{for any function $v\in C^\infty_c(\mathbb{R}^m)$ }\right\}$$
consists of so-called non-radiating sources. A non-radiating source of this type has the property that $f$ vanishes on the boundary of its support (which may have singularities).  Our analysis, on the other hand can be used to determine necessary local regularity properties for the boundary of the support of a non-radiating source, provided it satisfies a non-vanishing condition. The analysis leading to Theorem \ref{scatnon1} and Theorem \ref{scatnon2} implies the following results for the source problem (\ref{sourse}). 

\begin{theorem}\label{sonon1} Assume that $f=0$ in $\mathbb{R}^m \setminus \overline{D}$, $f|_{D} \in L^\infty(D)$ and that the  boundary $\partial D$ is Lipschitz. Suppose there exists $x_0\in\partial D$ such that $f(x_0)\neq 0$, and $f$ is real analytic in $\overline{D}\cap B_R(x_0)$ for some ball  $B_R(x_0)$ centered at $x_0$ of fixed radius $R>0$,  and  furthermore suppose  $\partial D\cap B_r(x_0)$ is not real analytic for any $r>0$.  Then the source $f$ radiates.
\end{theorem}
\noindent
In fact, for less regular sources we can prove a similar result.
\begin{theorem}\label{sonon2} Assume that $f=0$ in $\mathbb{R}^m \setminus \overline{D}$, $f|_{D} \in L^\infty(D)$ and that the  boundary $\partial D$ is Lipschitz. Suppose there exists $x_0\in\partial D$ such that $f(x_0)\neq 0$, and $f\in C^{m,\mu}(\overline{D}\cap B_R(x_0))\cap C^{1,1}(\overline{D}\cap B_R(x_0))$ for $m \geq 1$, $0<\mu<1$ for some ball  $B_R(x_0)$ centered at $x_0$ of fixed radius $R>0$,  and  furthermore suppose $\partial D\cap B_r(x_0)$ is not of class $C^{m+1,\mu}$ for any $r>0$.  Then the source $f$ radiates.
\end{theorem}

\appendix
\section{Appendix}
\subsection{Estimation of integrals (\ref{ji}) and (\ref{jm})}\label{ap1}
Our calculations here follow almost verbatim  \cite[page 363-364]{pom3}. We include these for the convenience of the reader, but we show only the calculations for the more complicated integral (\ref{jm}). After the change of variable $y^{(m-1)}=\eta u^{(m-1)}$, $f(y^{(m-1)})=|y^{(m-1)}|g(y^{(m-1)})$ the integrand in (\ref{jm}) takes the form
\begin{equation}\label{frac}
\frac{1}{\eta^{m-1}}\frac{(U-1)\left(|u^{(m-1)}|^2+(U+1)^2\right)^{m/2}-(U+1)\left(|u^{(m-1)}|^2+(U-1)^2\right)^{m/2}}{\left[\left(|u^{(m-1)}|^2+U^2+1\right)^2-4U^2\right]^{m/2}}
\end{equation}
with $U=|u^{(m-1)}|g(\eta u^{(m-1)})$.  The denominator of the (second) fraction in (\ref{frac}) is equal to 
$$\left(|u^{(m-1)}|^2+1\right)^{m}\left\{\left[1+\frac{|u^{(m-1)}|^2}{1+|u^{(m-1)}|^2}g^2(\eta u^{(m-1)})\right]^2-\left[\frac{2|u^{(m-1)}|}{1+|u^{(m-1)}|^2}g(\eta u^{(m-1)})\right]^2\right\}^{m/2}.$$
This can be estimated from below by using the fact that for all real numbers $a$ and $b$ we have 
\begin{equation}
\label{interm}
\left[1+\frac{a^2}{1+a^2}b^2\right]^2-\left[\frac{2a}{1+a^2}b\right]^2\geq \frac{4}{b^2+4}~,
\end{equation}
with equality at $a=\pm1/\sqrt{b^2+3}$. Indeed, we apply (\ref{interm}) with  $a:=|u^{(m-1)}|$ and $b:=g(\eta u^{(m-1)})$, and notice that $g(\eta u^{(m-1)})=|g(y^{(m-1)})|\leq K$ (independently of $\eta$, for $y^{(m-1)}\in B_\rho^{(m-1)}$) since $f$ is Lipschitz with constant $K$ and $f(0)=0$. As a consequence it follows that the denominator of (\ref{frac}) is bounded below by
$$\eta^{m-1}\left(\frac{4}{K^2+4}\right)^{m/2}\left(|u^{(m-1)}|^2+1\right)^{m}$$
With $a:=|u^{(m-1)}|$ and $b:=g(\eta u^{(m-1)})$, the numerator of (\ref{frac}) reads
$$(a^2+1)^{m/2}\left[(ab-1)\left(1+\frac{a^2b^2}{1+a^2}+\frac{2ab}{1+a^2}\right)^{m/2}-(ab+1)\left(1+\frac{a^2b^2}{1+a^2}-\frac{2ab}{1+a^2}\right)^{m/2}\right]~.$$
We now define $A:=1+\frac{a^2b^2}{1+a^2}$ and $B:=\frac{2ab}{1+a^2}$, both of which are obviously uniformly bounded in $\eta$ for $y^{(m-1)} \in B_\rho^{(m-1)}$, since $|b|=|g(\eta u^{(m-1)})|\leq K$. We can now write
$$\left((A\pm B)^{m/2}-A^{m/2}\right)\left((A\pm B)^{m/2}+A^{m/2}\right)=(A\pm B)^m-A^m=B{\mathcal P}_{\pm}(A,B)$$
where ${\mathcal P}_{\pm}(A,B)$ are polynomials on $A$ and $B$ of total order $m-1$. Noting that $1\le A \le K$  and $|B| \le K$ we get 
$$(A\pm B)^{m/2}=A^{m/2}+B{\mathcal P}_{\pm}(A,B)\left((A\pm B)^{m/2}+A^{m/2}\right)^{-1}=A^{m/2}+B{\mathcal Q}_{\pm}$$
where ${\mathcal Q}_{\pm}$ are uniformly bounded in $\eta$ for $y^{(m-1)} \in B_\rho^{(m-1)}$. The numerator thus becomes
\begin{eqnarray*}
&&(a^2+1)^{m/2}\left[(ab-1)(A^{m/2}+B{\mathcal Q}_+)-(ab+1)(A^{m/2}+B{\mathcal Q}_-)\right]\\
&&\hskip 50pt =(a^2+1)^{m/2}\left[ -2 A^{m/2}-B({\mathcal Q}_++{\mathcal Q}_-)+ abB({\mathcal Q}_+ -{\mathcal Q}_-)\right]~.
\end{eqnarray*}
Since we also have $|abB|\le K$ it follows, that in terms of the original notations, the absolute value of the numerator of (\ref{frac}) is bounded by
$$C \left(|u^{(m-1)}|^2+1\right)^{m/2}~,$$
uniformly in $\eta$ for $y^{(m-1)} \in B_\rho^{(m-1)}$.
Finally, combining the above estimates we obtain
$$\mbox{integrant in (\ref{jm})} \leq \frac{C}{\eta^{m-1}} \frac{\left(|u^{(m-1)}|^2+1\right)^{m/2}}{\left(|u^{(m-1)}|^2+1\right)^m}=\frac{C}{\eta^{m-1}} \frac{1}{\left(|u^{(m-1)}|^2+1\right)^{m/2}}$$
where $C>0$ stands for some positive constant independent of $u^{(m-1)}$. The desired estimate for the integral (\ref{jm}) now follows by the change of variables $y^{(m-1)}=\eta u^{(m-1)}$.

\noindent 
The bound for the integral (\ref{ji}) can be obtained in a similar way; we leave the details to the reader.

\subsection{An algebraic lemma}\label{ap2}
Below we provide the full statement of the algebraic lemma, which was used in Section \ref{specinc}. The lemma is taken directly from \cite{pom1}, where it appears as Lemma 6.1.  We refer to \cite{pom1} for a simple proof of this lemma.
\begin{lemma}
Let $a<b$ and $L>0$ be three real numbers. Let $\{ c_n\}_{n=1}^\infty$ be a monotonically increasing sequence of positive numbers tending to infinity and starting with $1<c_1$. Let $\{\mu_n\}_{n=1}^\infty$ be a sequence of positive numbers which satisfy $c_n \mu_n<\mu_{n+1}$. Given any $t \in \mathbb{R}$ there exists a number $s~:~ a<s<b$ such that 
$$
\mu_n s \rightarrow t ~ \hbox{ modulo } L ~ \hbox{ as } n \rightarrow \infty~.
$$
\end{lemma}

\section*{Acknowledgments}
{The research of FC was partially supported  by the AFOSR Grant  FA9550-20-1-0024 and  NSF Grant DMS-18-13492. The research of MSV was partially supported by NSF Grant DMS-12-11330. This work was carried out while MSV was visiting the University of Copenhagen and the Danish Technical University. This visit was in part supported by the Nordea Foundation and the Otto Mo \hskip -10pt /nsted Foundation.}


\begin{thebibliography}{10}
 
 \bibitem{pom0} Berenstein, C.A., An inverse spectral theorem and its relation to the Pompeiu problem, {\it J. d'Analyse Math.} {\bf 37}, 128-144 (1980).
 

\bibitem{nonscat4}
Bl{\aa}sten E., Nonradiating sources and transmission eigenfunctions vanish at corners and edges, {\it SIAM J. Math. Anal.} {\bf 50}, 6255-6270 (2018).


\bibitem{BLLW17}
Bl{\aa}sten E., Li X., Liu H. and Wang Y., On vanishing and localizing of transmission eigenfunctions near singular points: a numerical study,  {\em Inverse Problems} {\bf 33}, 105001, (2017).


\bibitem{liu}
Bl{\aa}sten E. and Liu H., Scattering by curvatures, radiationless sources, transmission eigenfunctions and inverse scattering problems,  {\it SIAM J. Math. Anal.} (to appear).

\bibitem{nonscat}
Bl{\aa}sten E., P\"aiv\"arinta  L. and Sylvester J., Corners always scatter, {\it Comm. Math. Phys.} {\bf 331}, 725-753 (2014).


\bibitem{pom00} Caffarelli L.A., The regularity of free boundaries in higher dimensions, {\it Acta Math.}  {\bf 139},  155-184 (1977).

\bibitem{caf2} Caffarelli L. A., The smoothness of the free surface in a filtration problem, {\it Arch. Rational Mech. Anal.} {\bf 63},  77-86 (1976). 

 \bibitem{CakoniColtonHaddar2016}  Cakoni F., Colton D. and Haddar H., \textit{Inverse Scattering Theory and Transmission Eigenvalues}, CBMS Series,  SIAM Publications, {\bf 88} 2016.

\bibitem{1}
Cakoni F., Gintides G. and Haddar H., The existence of an infinite discrete set of transmission eigenvalues, 
{\it SIAM J. Math. Anal.} {\bf 42}, 237-255 (2010).
 
 \bibitem{nn5} Cakoni F. and   Xiao J.,  On corner scattering for operators of divergence form and applications to inverse scattering,  {\it Communications in PDEs} {\bf 46}, 413-441 (2021).

\bibitem{harmonic}
Capogna L., Kenig E. and Lanzani L. {\it Harmonic Measure: Geometric and Analytic Points of View},  University Lecture Series, ULECT/35, American Mathematical Society 2005.


\bibitem{coltonkress}
Colton D. and Kress R., \textit{ Acoustic and Electromagnetic Scattering Theory}, Applied Mathematical Sciences {\bf 93}, Springer, New York, 4nd Ed. 2019.


\bibitem{dah}
 Dahlberg B. E. J.,  Estimates of harmonic measure, {\it Arch. Rational Mech. Anal.} {\bf 65}, 275-288 (1977).
 
 

\bibitem{ElH18}
Elschner J. and Hu G., Acoustic scattering from corners, edges and circular cones,{\it Arch. Ration. Mech. Anal.} {\bf 228}, 653-690 (2018).


\bibitem{HK}
Hayman W. K. and Kennedy P. B., \textit{Subharmonic functions Vol. I.},  London Mathematical Society Monographs  {\bf 9},  Academic Press, London-New York, 1976.

\bibitem{HSV16}
Hu G., Salo M. and  Vesalainen E.V., Shape identification in inverse medium scattering problems with a single far-field pattern, {\it SIAM J. Math. Anal.} {\bf 48}, 152-165 (2016).
 

\bibitem{hw}
 Hunt A. and Wheeden R. L., On the boundary values of harmonic functions, {\it Trans. Amer. Math. Soc.} {\bf 132}, 307-322 (1968). 

\bibitem{pom000} Kinderlehrer D. and Nirenberg L., Regularity in free boundary problems, {\it Ann. Scuola Norm. Sup. Pisa Cl. Sci.} Serie IV, {\bf 4}, 373-391 (1977).

\bibitem{kirsch} Kirsch A. and Hettlich F., \textit{The Mathematical Theory of Time-Harmonic Maxwell's Equations}, Applied Mathematical Sciences {\bf 190}, Springer, New York, 2015.

\bibitem{kusiak} 
Kusiak S.  and Sylvester J., The scattering support, {\it Comm, Pure Appl. Math.}  {\bf 56},  1525-1548 (2003). 

\bibitem{res1}  Melrose R.B.,  \textit{Geometric Scattering Theory}, Cambridge University Press, Cambridge, 1995.

\bibitem{nguyen} Nguyen H-M. and Nguyen Q-H., The Weyl law of transmission eigenvalues and the completeness of generalized transmission eigenfunctions, arXiv:2008.08540 (to appear).


\bibitem{PSV17}
P\"aiv\"arinta L., Salo M. and Vesalainen E., Strictly convex corners scatter. {\it Rev. Mat. Iberoam.} {\bf 33}, 1369--1396  (2017).

\bibitem{vodev}
Vodev G., High-frequency approximation of the interior Dirichlet-to-Neumann map and applications to the transmission eigenvalues, {\it  Anal. PDEs}  {\bf 11}, 213-236 (2018).

\bibitem{pom1}
Vogelius M., An inverse problem for the equation $\Delta u=-cu-d$, {\it Annales de l'Institut Fourier} {\bf 44}, 1181-1209 (1994).

\bibitem {vx} Vogelius M. and Xiao J., Finiteness results concerning non-scattering wave numbers for incident plane- and Herglotz waves, {\it SIAM J. Math. Anal.} (to appear).




\bibitem{pom2}  Williams S.A., A partial solution of the Pompeiu problem, {\it Math. Ann.} {\bf 223}, 183-190 (1976).

\bibitem{pom3}  Williams S. A., Analyticity of the boundary for Lipschitz domains without the Pompeiu property, {\it Indiana Univ. Math. J.} {\bf 30}, 357-369 (1981).


 \end{thebibliography}
\end{document}